\documentclass[11pt]{amsart}
\usepackage{amsfonts}
\usepackage{amsmath}
\usepackage{amsthm}
\usepackage{url}
\usepackage[all]{xy}
\usepackage{latexsym}
\usepackage{amssymb}
\usepackage[cp850]{inputenc}
\usepackage{amsfonts}
\usepackage{graphicx}
\usepackage{hyperref} 
\usepackage[capitalise]{cleveref}
\usepackage{eucal} 
\usepackage{nccmath} 
\usepackage{units}   
\usepackage{bookmark}
\usepackage{mathtools}

\usepackage[left=2.5cm,right=2.5cm,top=3cm,bottom=3cm]{geometry}

\usepackage{dsfont}
\usepackage{float}

\usepackage[usenames]{color}
\usepackage{xfrac}

\usepackage{setspace}
\newtheorem{theo}{Theorem}[section]
\newtheorem*{theo*}{Theorem}
\newtheorem{lem}[theo]{Lemma}
\newtheorem{definition}[theo]{Definition}

\newtheorem{cor}[theo]{Corollary}
\newtheorem{propo}[theo]{Proposition}
\newtheorem*{rmq}{Remark}

\let\ssection=\section
\renewcommand{\section}{\setcounter{equation}{0}\ssection}

\newtheorem*{namedtheorem}{\theoremname}
\newcommand{\theoremname}{Theorem}
\newenvironment{named}[1]{\renewcommand{\theoremname}{#1}\begin{namedtheorem}}{\end{namedtheorem}}

\newcommand{\BR}{\mathbb R}			
\newcommand{\BN}{\mathbb N}			
			\newcommand{\BZ}{\mathbb Z}

\newcommand{\CA}{\mathcal A}		
\newcommand{\CC}{\mathcal C}

		\newcommand{\CL}{\mathcal L}
\newcommand{\CM}{\mathcal M}

\newcommand{\Fm}{\mathfrak m}

\newcommand{\Fc}{\mathfrak c}

\newcommand{\FC}{\mathfrak C}

\newcommand{\D}{\partial}

\DeclareMathOperator{\Map}{Map}

\DeclareMathOperator{\height}{height}

\newcommand{\comment}[1]{}

\DeclareMathOperator{\supp}{Supp}

\newcommand{\fsubd}{\mathrel{{\scriptstyle\searrow}\kern-1ex^d\kern0.5ex}}
\newcommand{\bsubd}{\mathrel{{\scriptstyle\swarrow}\kern-1.6ex^d\kern0.8ex}}

\renewcommand{\epsilon}{\varepsilon}

\newcommand*{\newfaktor}[2]{
  \raisebox{0.5\height}{\ensuremath{#1}}
  \mkern-5mu\diagup\mkern-4mu
  \raisebox{-0.5\height}{\ensuremath{#2}}
}

\begin{document}

\title[]{Counting arcs of the same type}
\author{Marie Trin}
\address{UNIV RENNES, CNRS, IRMAR - UMR 6625, F-35000 RENNES, FRANCE}
\email{marie.trin@univ-rennes.fr}
\thanks{The author is supported by Centre Henri Lebesgue, program ANR-11-LABX-0020-0 and R\'{e}gion Bretagne }
\date{\today}

\begin{abstract} We prove a general counting result for arcs of the same type in compact surfaces. We also count infinite arcs in cusped surfaces and arcs in orbifolds. These theorems are derived from a result that guarantees the convergence of certain measures on the space of geodesic currents.
\end{abstract}

\maketitle

\section{Introduction}

The mapping class group $\Map(\Sigma)$ of a connected oriented surface $\Sigma$ with genus $g$ and $r$ boundary components acts on the set of weighted multicurves $\FC_m(\Sigma)$. The question of counting the elements in a given orbit has been studied by M. Mirzakhani for simple curves \cite{Mir1} and later for general curves \cite{Mir2}. She proved that for any complete finite area hyperbolic metric $X$ on $\Sigma$, any weighted multicurve $\gamma_0\in\FC_m(\Sigma)$ and any finite index subgroup $\Gamma$ of $\Map(\Sigma)$ there is a constant $\Fc^\Gamma_{g,r}(\gamma_0)$ such that
\begin{equation}
\lim\limits_{L\to\infty}\dfrac{\#\{\gamma\in\Gamma\cdot\gamma_0 | \ell_X(\gamma)\leq L \}}{L^{6g-6+2r}}=\Fc^\Gamma_{g,r}(\gamma_0)\cdot \Fm^{\Sigma}_{Thu}(\{\ell_X(\cdot)\leq 1\}),
\end{equation}
where $\Fm^{\Sigma}_{Thu}$ is the Thurston measure on the measured geodesic laminations of $\Sigma$.

Erlandsson-Souto \cite{ES} have extended this theorem into a general version where the hyperbolic length function can be replaced by other notions of complexity for the curves of $S$: that is any positive, continuous and homogeneous function on the geodesic currents of $S$ -- this applies for example to the length for any Riemannian metric on $\Sigma$ or its interior \cite{ES2}, the intersection number with a filling curve or current \cite{ES2}\cite{RS}, the word length \cite{VE} or the translation length in $\pi_1(\Sigma)$ when acting on a metric space \cite{EPS}...

Now, if $\Sigma$ is a compact connected oriented surface with non-empty boundary then one can consider the action of the mapping class group on the set of weighted multiarcs. In this setting, N. Bell \cite{Bell22} proved a result close to Mirzakhani's: if $X$ is a complete hyperbolic metric with geodesic boundary on $\Sigma$ then for every weighted multiarc $\alpha_0$ there is a constant $\Fc^{\Map}_{g,r}(\alpha_0)$ such that
\begin{equation} \label{Bell}
\lim\limits_{L\to\infty}\dfrac{\#\{\alpha\in\Map(\Sigma)\cdot\alpha_0 | \ell_X(\gamma)\leq L \}}{L^{6g-6+2r}}=\Fc^{\Map}_{g,r}(\gamma_0)\cdot\Fm^{\Sigma}_{Thu}(\{\ell_X(\cdot)\leq 1\}),
\end{equation}

where the length of an arc $\alpha$ is the length of the unique geodesic arc homotopic to $\alpha$ which is orthogonal to the geodesic boundary. The main goal of this paper is to obtain a general version of that result by proving a convergence result for a certain sequence of measures. Indeed, already Erlandsson-Souto generalization of Mirzakhani's results relies on the convergence of certain measures on the space of geodesic currents. Here it is key that curves can be seen as currents. To see arcs as currents we will work in the doubled surface $D\Sigma$ of $\Sigma$. Denoting by $\widehat{\alpha}$ the curve in $D\Sigma$ corresponding to the doubling of an arc $\alpha$ of $\Sigma$, we define for any weighted multiarc $\alpha_0$ and for any finite index subgroup $\Gamma$ of $\Map(\Sigma)$ the Radon  measures \begin{align}
	\nu^\Gamma_{\alpha_0,L}=\dfrac{1}{L^{6g-6+2r}}\sum\limits_{\alpha\in \Gamma\cdot\alpha_0}\delta_{\frac{1}{L}\widehat{\alpha}}
\end{align}
on the space $\CC(D\Sigma)$ of geodesic currents of $D\Sigma$. We prove that they converge when $L$ tends to infinity. The following is our main theorem.

\begin{named}{Theorem \ref{Main}}
	If $\Sigma$ is a compact connected oriented surface with non-empty boundary and negative Euler characteristic which is not a pair of pants, then for every weighted multiarc $\alpha_0 \in \CA_m(\Sigma)$, and every finite index subgroup $\Gamma$ of $\Map(\Sigma)$, there is $ \Fc^\Gamma_{g,r}(\alpha_0)>0$ such that
	\begin{equation*}
	\lim\limits_{L\to\infty}\nu^\Gamma_{\alpha_0,L} =\Fc^\Gamma_{g,r}(\alpha_0)\cdot\widehat{\Fm}_{Thu}^\Sigma. 
	\end{equation*}
	Here $\widehat{\Fm}_{Thu}^\Sigma$ is a Radon measure on $\CC(D\Sigma)$ and the convergence occurs with respect to the weak* topology on the set of Radon measures on  $\CC(D\Sigma)$.
\end{named}

\begin{rmq}
	In \cref*{Main}, the measure $\widehat{\Fm}_{Thu}^\Sigma$ is a specific measure on $\CC(D\Sigma)$ obtained from the Thurston measure on the space $\CM\CL(\Sigma)$ of measured laminations on $\Sigma$, see Section 2.2 for details.
\end{rmq}

We will get from \cref*{Main} a pretty general counting theorem for arcs. We will count arcs with bounded complexity where the complexity of an arc is given by functions on $\CA_m(\Sigma)$. We will say that such a function $F$ \textit{extends to currents} if there exists a continuous and homogeneous function on $\CC(D\Sigma)$ whose restriction to the set of arcs is $F$. If this function on currents is also called $F$, it means that $F(\widehat{\alpha})=2F(\alpha)$ for every arc $\alpha$. Since $\CC(\Sigma)\subset \CC(D\Sigma)$ we will say that $F$ is positive if it is a positive function on $\CC(\Sigma)$.

\begin{named}{Corollary \ref{Corollaire}}

		Let $\Sigma$ and $\Gamma$ be as in \cref*{Main}. For any weighted multiarc $\alpha_0\in\CA_m(\Sigma)$ and any function $F$ on $\CA_m(\Sigma)$ which extends to a positive function on currents we have 
		$$\lim\limits_{L\to\infty}\frac{\#\{\alpha\in\Gamma\cdot\alpha_0|F(\alpha)\leq L \}}{L^{6g-6+2r}}= \Fc^\Gamma_{g,r}(\alpha_0)\cdot\Fm_{Thu}^\Sigma(\{F(\cdot)\leq 1\}). $$
\end{named}

For example, $F$ can be the length function for any Riemannian metric with geodesic boundary on $\Sigma$ or the intersection number with a filling curve or current of $\Sigma$. 
\vspace*{0.3cm}

In Section 2, we will recall the needed background on geodesic currents and explain how we double $\Sigma$. Section 3 will be dedicated to the proof of a first counting theorem, \cref*{mainTheo1}, with a condition on intersection numbers. That theorem will be a key tool in order to prove \cref*{Main} in Section 4. In the last section we will prove \cref*{Corollaire}, we will also obtain counting results for bi-infinite arcs (\cref{CuspToCusp}) or for arcs on orbifolds (\cref{MainOrbifold}).
\vspace{0.5cm}

\textbf{Acknowledgements.} I am grateful to Viveka Erlandsson for our discussions which have initiated the work on this paper and to Nick Bell who introduced to me his method for counting arcs. I also want to thank Juan Souto for all our conversations and his suggestions related to the topic of this paper.

\section{Geodesic currents}
In this section, we describe some background on geodesic currents, explaining how currents of $\Sigma$ can be seen as currents of $D\Sigma$. For more details on the notion of geodesic currents, we refer to \cite{AL}, \cite{Bon} or \cite{ES}.

\subsection{Background}

Recall that for any (compact) connected, oriented hyperbolic surface $X$, a \textit{geodesic current} is a $\pi_1(X)$-invariant Radon measure on the set of bi-infinite unoriented geodesics of the universal cover $\widetilde{X}$. Note that all the hyperbolic structures on a given topological surface $\Sigma$ define the same set of geodesic currents. Hence, for $\Sigma$ a connected oriented surface with negative Euler characteristic we denote by $\CC(\Sigma)$ the space of geodesic currents of $\Sigma$. This space is endowed with the weak* topology and is then Hausdorff, metrizable, and second countable. Moreover, if the surface $\Sigma$ is compact then $\CC(\Sigma)$ is locally compact. The two main examples of geodesic currents we are interested in are weighted-multicurves and measured laminations.
\vspace{0.3cm}

 By a \textit{curve} we mean a free homotopy class of essential (\textit{ie.} non-null-homotopic and non-peripheral) closed curves. A \textit{weighted multicurve} is a formal finite sum of different curves with positive weights.  When a metric is fixed on $\Sigma$, a curve is canonically represented by its unique geodesic representative. Hence, it lifts to the universal cover into a discrete $\pi_1(\Sigma)$-invariant set of bi-infinite geodesics and the counting measure over this set is a geodesic current. The geodesic currents associated to a weighted multicurve is the corresponding sum of currents. In the following, we will denote by $\FC(\Sigma)$ and $\FC_m(\Sigma)$ the sets of curves and weighted multicurves of $\Sigma$. They will mostly be seen as subsets of $\CC(\Sigma)$.
\vspace{0.3cm}

The set $\CM\CL(\Sigma)$ of \textit{measured laminations} can also be seen as a subset of $\CC(\Sigma)$. Recall that a geodesic lamination is a closed subset of the interior of $\Sigma$ that can be foliated by disjoint simple geodesics. This definition ensures that the geodesic boundaries can-not be leaves of a lamination. A measured lamination is a geodesic lamination endowed with a transverse measure and it is that transverse measure that allows us to see measured laminations as geodesic currents. A point to notice is that, as geodesic currents of a closed surface, the measured laminations are characterised by having zero self intersection number. In the case of a surface with boundary, we have to add the condition that the mesured lamination give no weight to the boundary components:
\begin{equation} \label{MLinC}
	\CM\CL(\Sigma)=\{\mu\in\CC(\Sigma) | i(\mu,\mu)=0 \quad \text{and} \quad \mu(\widetilde{\D\Sigma})=0\}.
\end{equation}
\vspace{0.3cm}

 We will work with Radon measures on the space of currents. The one we will be mainly interested in is the \textit{Thurston measure} $\Fm_{Thu}^\Sigma$ on $\CM\CL(\Sigma)$ wich extends naturally to a measure on $\CC(\Sigma)$. Although we will not use its precise expression in the following, let us recall that this measure is given by
\begin{equation*}
	\Fm_{Thu}^\Sigma = \lim\limits_{L\to\infty} \frac{1}{L^{6g-6+2r}}\sum\limits_{\gamma\in\CM\CL_\BZ(\Sigma)}\delta_{\frac{1}{L}\gamma},
\end{equation*}
where $\CM\CL_\mathbb{Z}(\Sigma)$ is the set of integral weighted simple multicurves. For more details on the Thurston measure one can refer to \cite{MoTe} or \cite{FAH}.
\vspace*{0.4cm}

\subsection{Doubling the surface}

In the following, $\Sigma$ is a compact connected oriented surface with $r>0$ boundary components and genus $g$ such that $2-2g-r<0$ and $(g,r)\neq (0,3)$. We endow it with a fixed hyperbolic metric with geodesic boundary. Note that the orientation on $\Sigma$ induces an orientation of the geodesic boundary components. \textit{An arc} $\alpha$ in $\Sigma$ is a free homotopy class of oriented segments based on boundary components and we identify two arcs that deffer from the orientation. A \textit{weighted multiarc} $\alpha$ is a finite sum of arcs with positive weights. We ask the arcs not to be homotopic to a segment of a boundary component. We will denote by $\CA(\Sigma)$ the set of arcs and $\CA_m(\Sigma) $ the set of weighted multiarcs. Note that since a metric is fixed, an arc $\alpha$ is canonically represented by the unique orthogeodesic of the homotopy class. 
\vspace{0.3cm}

As we mentioned in the introduction, we need to be able to interpret arcs as currents. However, it is not possible to do it by using the same process as for curves. This is why we will work on the doubled surface $D\Sigma$. The surface $D\Sigma$ is the closed oriented surface of genus $g(D\Sigma)=2g+r-1$ corresponding to the doubling of $\Sigma$. In that setting, the arcs of $\Sigma$ will be in bijection with certain symmetric curves of $D\Sigma$.

\begin{figure}[!ht]
	\centering
	\includegraphics[scale=0.7]{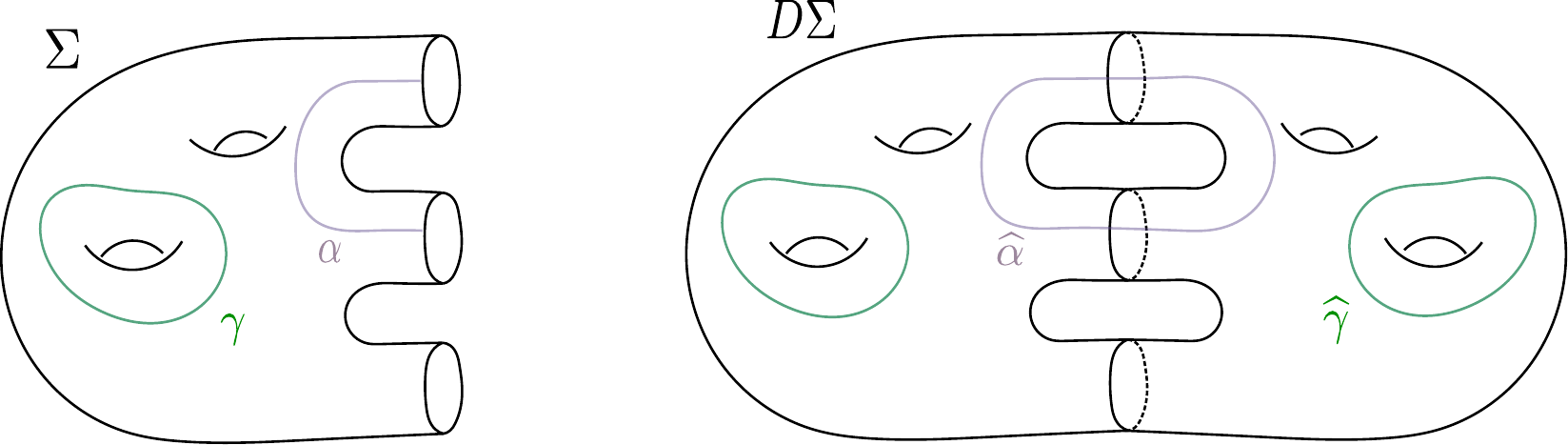}  
	\caption{From $\Sigma$ to $D\Sigma$}
	\label{doubling}
\end{figure}

We can embed two copies of $\Sigma$ into $D\Sigma$ such that they cover $D\Sigma$ and meet pointwise along their boundary components. We will denote by $\Sigma^+$ and $\Sigma^-$ these two copies, $i^+$ and $i^-$ the associated embeddings, and $\sigma : D\Sigma \longmapsto D\Sigma$ the involution that exchanges $\Sigma^+$ and $\Sigma^-$ (it is an orientation reversing map which is the identity when restricted to the boundary of $\Sigma$). The embeddings $i ^+$ and $i^-$ naturally extend to embdeddings from the geodesic currents of $\Sigma$ to the geodesic currents of $D\Sigma$ hence, for an element $\mu$ in $\CC(\Sigma)$ we will denote by $\widehat{\mu}$ its doubled version:
\begin{equation}\label{hat}
	\begin{array}{ccccl}
	\widehat{\cdot} & : & \CC(\Sigma) & \hookrightarrow & \CC(D\Sigma) \\
	& & \mu & \mapsto &  i^+(\mu) + i^-(\mu). 
	\end{array}
\end{equation}

To obtain \cref{Main}, we need to see the Thurston measure as a Radon measure on  $\CC(D\Sigma)$. To do so, we pushfoward the measure through the hat operator to obtain a mesure $\widehat{\Fm}_{Thu}^\Sigma$ in $\CC(D\Sigma)$ surpported by $\widehat{\CM\CL}(\Sigma).$

Remark that the elements in the image of the hat operator are fixed by the involution $\sigma$. More generally, we will call \textit{symmetric} the elements fixed by $\sigma$ and we will denote by $\FC^\sigma(\D\Sigma)$ the symmetric curves of $D\Sigma$ and $\CM\CL^\sigma(D\Sigma)$ the symmetric measured laminations. Remark that the set of symmetric measured laminations is larger than the image of $\CM\CL(\Sigma)$ by the hat operator: the embedded boundary components of $\Sigma$ are symmetric but are not represented by elements in the image of $\CM\CL(\Sigma)$ by the hat operator. We record this fact for later reference.

\begin{propo}\label{caracterisation}
	A symmetric measured lamination $\Lambda\in\CM\CL^\sigma(D\Sigma)$ is an element of $\widehat{\CM\CL}(\Sigma)$ if and only if
	\begin{enumerate}
		\item it does not have connected components of $\D\Sigma$ as leaves, 
		\item $i(\Lambda,\partial \Sigma)=0$, where $i(\cdot,\cdot)$ is the intersection form between currents. \hspace*{3.4cm}\qedsymbol
	\end{enumerate}
\end{propo}

At last, for any (multi)arc $\alpha\in\CA(\Sigma)$ its two copies $i^+(\alpha)$ and $i^-(\alpha)$ into $D\Sigma$ meet at their endpoints and their union forms a symmetric (multi)curve of $D\Sigma$. We will denote by $\widehat{\alpha}$ that curve: it is not an image by the above hat operator but this notation is consistant with the one for curves or measured lamination as their image through $\widehat{\cdot}$ are the union of their two copies (see \cref{doubling}). The curves of $D\Sigma$ are geodesic currents of $D\Sigma$, so, the doubling process implies that we are now able to see arcs as geodesic currents.	

\begin{equation}\label{hatArcs}
	\begin{array}{cccclcl}
	\widehat{\cdot} & : & \CA_m(\Sigma) & \hookrightarrow & \FC_m^\sigma(D\Sigma) & \hookrightarrow& \CC(D\Sigma) \\
	& & \alpha & \mapsto &  \widehat{\alpha} & \mapsto &\widehat{\alpha}.
	\end{array}
\end{equation}

\section{Counting problems with a bound on the intersection number}

Our next goal is to prove a particular version of \cref{Corollaire}, namely the fact that we can count arcs when we measure them using the intersection with a filling curve. Our argument is inspired by those of Bell \cite{Bell22}\cite{BellThesis}. Bell's approach consists of associating to each arc $\alpha$ the curve $\gamma_\alpha = \alpha^{-1} \cdot a_2 \cdot \alpha\cdot a_1$, where $a_2$ and $a_1$ are the boundary components at the end and begining of $\alpha$ ---  whatever the chosen orientation for $\alpha$, the associated curve is the same. It turns out that $\alpha$ and $\gamma_\alpha$ are closely related and we will be able to extend the counting results for $\gamma_\alpha$ to results for $\alpha$.

The above construction of $\gamma_\alpha$ for a given curve induces a map from the set of weighted multiarcs to the set of weighted multicurves of $\Sigma$:
\begin{center}
	\begin{tabular}{cccrcl}
		$I$ &:& $\CA_m(\Sigma)$ & $\to$ & $\FC_m(\Sigma)$ \\
		& & $\sum a_i\alpha_i$ & $\mapsto$ & $\sum a_i \gamma_{\alpha_i}$.	
	\end{tabular}
\end{center}

The first point to notice about the map $I$ is that it is equivariant with respect to the mappping class group, meaning that for any $\phi \in\Map(\Sigma)$ and $\alpha\in\CA_m(\Sigma)$
\begin{equation*}
	\phi\cdot I(\alpha) = I(\phi\cdot\alpha).
\end{equation*}

Secondly, we can prove that $\alpha$ and $I(\alpha)$ are nearby in the sense that they intersect curves essentially in the same way.

\begin{lem} \label{Lemme}
	For any $\alpha\in\CA_m(\Sigma)$ there is $d_\alpha\in\BN$ such that if $\mu\in\FC_m(\Sigma)$ is a weighted-multicurve then  
	\begin{equation} \label{IneqInter}
		| i(I(\alpha),\mu)-2i(\alpha,\mu)|\leq 2d_\alpha i(\mu,\mu)
	\end{equation}
where $i(\cdot,\cdot)$ is the geometric intersection number.
Moreover, $d_\alpha$ is invariant under the action of the mapping class group on $\CA_m(\Sigma)$.
\end{lem}

\begin{proof}
		For any arc $\alpha\in\CA(\Sigma)$ there is an immersion, unique up to homotopy, that sends the pair of pants $P$ into $\Sigma$ in such a way that the image of the boundary components are $a_1$, $a_2$ and $I(\alpha)$, and such that $\alpha$ is the image of the unique simple arc between the preimage of $a_1$ and $a_2$.

	Let $H$ be the subgroup of $\pi_1(\Sigma)$ given by the image of $\pi_1(P)$ under the immersion. The group $H$ is the free group of rank 2 and the pair of pants lifts homeomorphically to $\newfaktor{\widetilde{\Sigma}}{H}$ as a compact subsurface. Surface groups being LERF \cite{scott}, there is a finite index subgroup $K$ of $\pi_1(\Sigma)$ containing $H$ such that $P$ lifts to $\newfaktor{\widetilde{\Sigma}}{K}$. This means that there is a cover of $\Sigma$, of degree $d_\alpha<\infty$, in which some well chosen lifts $\tilde{a}_1$, $\tilde{a}_2$, and $\widetilde{I(\alpha)}$ of $a_1$, $a_2$, and $I(\alpha)$ are the three boundary components of an embedded pair of pants and such that the unique simple arc between $\tilde{a}_1$ and $\tilde{a}_2$ is a lift $\tilde{\alpha}$ of $\alpha$ (see \cref{lift}).  If $\mu$ is a weighted multicurve of $\Sigma$ we denote by $\tilde{\tilde{\mu}}$ its preimage inside this cover.
	\vspace*{0.5cm}
	
	\begin{figure}[!ht]
		\centering
		\includegraphics[scale=0.4]{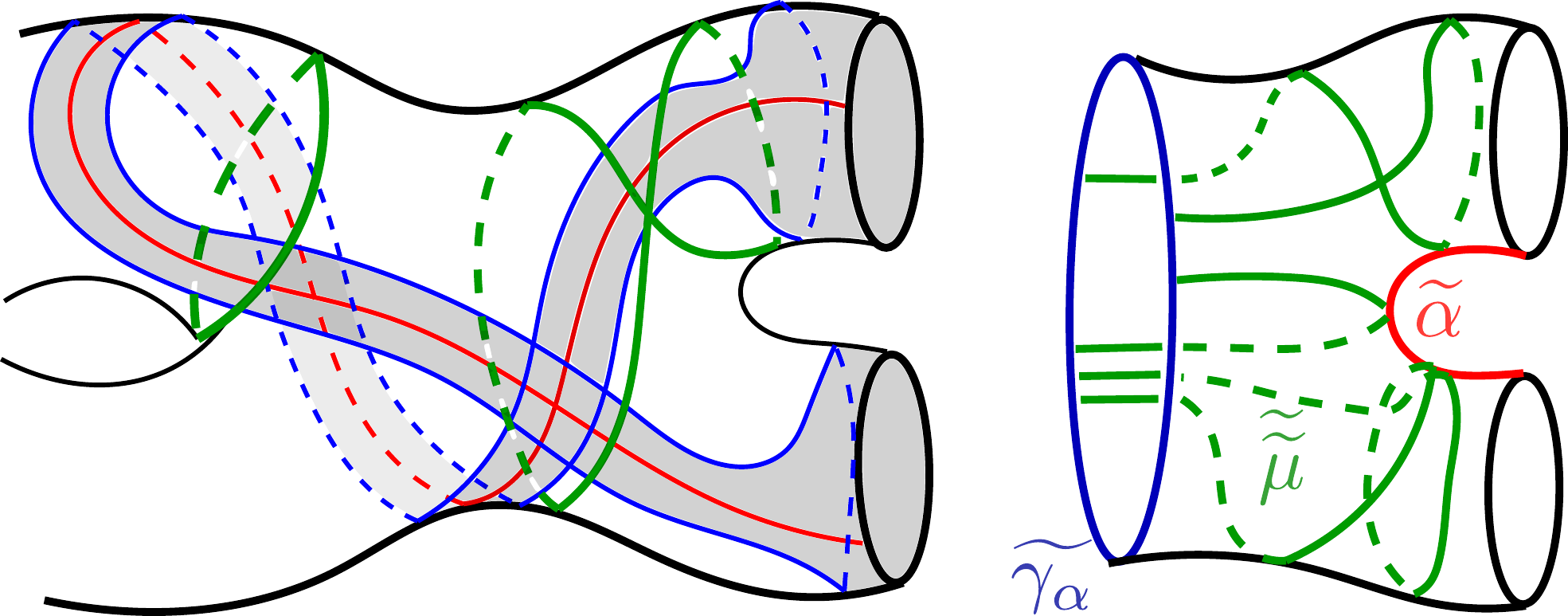}  
		\caption{Immersed and embedded pair of pants}
		\label{lift}
	\end{figure}
	\vspace*{0.5cm}

	For any finite degree cover of $\Sigma$, if $\tilde{\alpha}$ is a lift of $\alpha$, then there is a lift $\widetilde{I(\alpha)}$ of $I(\alpha)$ equal to $I(\tilde{\alpha})$ and it is the case for the lifts chosen above. So, in the previous cover we have the following relations between intersection numbers:
	\begin{enumerate}
		\item $i(I(\alpha),\mu)\leq 2i(\alpha,\mu)$ by construction of $I(\alpha)$,
		\item $i(I(\tilde{\alpha}),\tilde{\tilde{\mu}}) =  i(I(\alpha),\mu)$, $i(\tilde{\alpha},\tilde{\tilde{\mu}})= i(\alpha,\mu)$ and $i(\tilde{\tilde{\mu}},\tilde{\tilde{\mu}})=d_\alpha\cdot i(\mu,\mu)$ by definition of a covering map,
		\item $i(\tilde{\alpha},\tilde{\tilde{\mu}}) \leq \frac{1}{2}\cdot i(I(\tilde{\alpha}),\tilde{\tilde{\mu}}) + i(\tilde{\tilde{\mu}},\tilde{\tilde{\mu}})$ because each intersection between $\tilde{\alpha}$ and $\tilde{\tilde{\mu}}$ comes from a component of $\tilde{\tilde{\mu}}$ that enters and leaves the pair of pants, or a component that turns around the legs and which have some self intersections. A self intersection leads to at most one intersection with $\tilde{\alpha}$ and a pair of intersections with the boundary (\textit{ie} intersections with $I(\tilde{\alpha})$) to at most one intersection with $\tilde{\alpha}$.
	\end{enumerate}

All in all, it occurs that
\begin{equation}
	-2\cdot d_\alpha\cdot i(\mu,\mu) \leq 0 \leq 2\cdot i(\alpha,\mu)-i(I(\alpha),\mu) \leq  2\cdot d_\alpha\cdot  i(\mu,\mu),
\end{equation}
 and \cref{IneqInter} follows for $\alpha$.

Moreover, since any mapping class $\phi$ induces a bijection $\phi_* : \pi_1(\Sigma)\to\pi_1(\Sigma)$ we can choose $d_\alpha$ to be the same for every arc in a given orbit. We have proved the lemma for arcs and the triangle inequality gives the results for weighted multiarcs.\end{proof}

We can consider the restriction of $I$ to the orbit of a given weighted multiarc $\alpha_0$:
\begin{center}
	\begin{tabular}{cccrcl}
		$I_{|\Map(\Sigma)\cdot\alpha_0}$ &:& $\Map(\Sigma)\cdot\alpha_0$ & $\to$ & $\Map(\Sigma)\cdot I(\alpha_0)$ \\
		& & $\alpha$ & $\mapsto$ & $I(\alpha)$,
	\end{tabular}
\end{center}
equivariance under $\Map(\Sigma)$ and \cref{Lemme}  imply that this map is finite-to-one.

\begin{propo}\label{kalpha0}
	For all $\alpha_0\in\CA_m(\Sigma)$, the map $I_{|\Map(\Sigma)\cdot\alpha_0}$ is well defined and $k(\alpha_0)$-to-$1$ for some $k(\alpha_0)\in\BN$ which depends only on the type of $\alpha_0$. \hspace*{8cm}\qedsymbol
\end{propo}

We are now able to count arcs with respect to the intersection number with a curve. More precisely, we will count with respect to the intersection number with a \textit{filling multicurve} -- a curve that cuts the surface into disks and annulus -- which ensures that we count finitely many arcs at each step.

\begin{theo}\label{mainTheo1}
	If $\Sigma$ is a compact connected oriented surface with non-empty boundary and negative Euler characteristic which is not a pair of pants, and $\Gamma$ is any finite index subgroup of $\Map(\Sigma)$, then for any weighted multiarc $\alpha_0$ on $\Sigma$ and for any $\mu\in\FC_m(\Sigma)$ filling multicurve we have
	\begin{equation*}
		\lim\limits_{L\to\infty}\dfrac{\#\{\alpha\in \Gamma\cdot \alpha_0|i(\mu,\alpha)\leq L\}}{L^{6g-6+2r}}=\Fc^\Gamma_{g,r}(\alpha_0)\cdot\Fm_{Thu}^\Sigma(\{i(\mu,\cdot)\leq 1 \}).
	\end{equation*}
Here, $\Fc^\Gamma_{g,r}(\alpha_0)$ is a constant  fixed by the type of $\alpha_0$, the group $\Gamma$ and the topology of $\Sigma$.
\end{theo}

\begin{proof} The intersection number with a filling multicurve being a positive, homogenuous and continuous function on the geodesic currents of $\Sigma$, \cite[Theo. 9.1]{ES} or \cite[Ex. 9.1]{ES} ensure that there exists $\Fc^\Gamma_{g,r}(I(\alpha_0))>0$ such that 
	\begin{equation}\label{CasCourbe}
		\lim\limits_{L\to\infty}\dfrac{\#\{\gamma\in \Gamma\cdot I(\alpha_0)|i(\mu,\gamma)\leq L\}}{L^{6g-6+2r}}=\Fc^\Gamma_{g,r}(I(\alpha_0))\cdot\Fm_{Thu}^\Sigma(\{i(\mu,\cdot)\leq 1 \}).
	\end{equation} 
Hence, by \cref{Lemme} and \cref{CasCourbe} we have 
\begin{align*}
	\limsup_L & \frac{\# \{ \alpha\in\Gamma\cdot\alpha_0|i(\alpha,\mu) \leq L\}}{L^{6g-6+2r}} \\
	&\leq  k_{\alpha_0}\cdot\limsup_L\dfrac{\#\{\gamma\in\Gamma\cdot I(\alpha_0)| i(\mu,\gamma)\leq 2L+2d_{\alpha_0}i(\mu,\mu) \}}{L^{6g-6+2r}} \\
	& = k_{\alpha_0}\cdot2^{6g-6+2r}\cdot\limsup_L \dfrac{\#\{\gamma\in\Gamma\cdot I(\alpha_0)| i(\mu,\gamma)\leq 2L+2d_{\alpha_0}i(\mu,\mu) \}}{(2L+2d_{\alpha_0}i(\mu,\mu))^{6g-6+2r}}(1+\dfrac{d_{\alpha_0}i(\mu,\mu)}{L})^{6g-6+2r}\\
	&= k_{\alpha_0}\cdot2^{6g-6+2r}\cdot\Fc^\Gamma_{g,r}(I(\alpha_0))\cdot\Fm_{Thu}^\Sigma(\{i(\mu,\cdot)\leq 1 \}),
\end{align*}
where $k_{\alpha_0}$ comes from \cref{kalpha0}. With the same computations 
\begin{align*}
	\liminf_L & \frac{\# \{ \alpha\in\Gamma\cdot\alpha_0|i(\alpha,\mu) \leq L\}}{L^{6g-6+2r}} \geq k_{\alpha_0}\cdot2^{6g-6+2r}\cdot\Fc^\Gamma_{g,r}(I(\alpha_0))\cdot\Fm_{Thu}^\Sigma(\{i(\mu,\cdot)\leq 1 \}),
\end{align*}
and we obtain \cref{mainTheo1} with $		\Fc^\Gamma_{g,r}(\alpha_0)=k_{\alpha_0}\cdot2^{6g-6+2r}\cdot\Fc^\Gamma_{g,r}(I(\alpha_0))$.
\end{proof}

\section{Proof of the main theorem}

In this section, $\Sigma$ is still a compact connected oriented surface, with genus $g$, and $r>0$ boundary components, with negative Euler characteristic and such that $(g,r)\neq(3,0)$. For technical reasons, $\Sigma$ is endowed with a hyperbolic structure with geodesic boundary. Note that it induces a hyperbolic structure on $D\Sigma$. In the following, we fix a weighted multiarc $\alpha_0\in\CA_m(\Sigma)$ and a finite index subgroup $\Gamma$ of $\Map(\Sigma)$. The doubling process allows us to see $\alpha_0$ as a current (see \cref{hatArcs}) and to define a sequence $(\nu^\Gamma_{\alpha_0,L})_{L>0}$ of Radon measures on $\CC(D\Sigma)$ from $\alpha_0$ by

\begin{align}
    \nu^\Gamma_{\alpha_0,L}=\dfrac{1}{L^{6g-6+2r}}\sum\limits_{\alpha\in \Gamma\cdot\alpha_0}\delta_{\frac{1}{L}\widehat{\alpha}} \quad \forall L>0.
\end{align}

The strategy to prove \cref{Main} is the following. We will prove that $(\nu_{\alpha_0,L}^\Gamma)_{L>0}$ has accumulation points and that they are all supported by $\widehat{\CM\CL}(\Sigma)$. Afterwards, we will use the characterisation of $\Map(\Sigma)$-invariant measures on measured geodesic laminations given by Lindenstrauss-Mirzakhani to show that these accumulation points are all multiples of the pushforward by the hat operator of the Thurston measure on $\Sigma$. We will conclude proving that they are all the same multiple of $\widehat{\Fm}_{Thu}^\Sigma$. Note that at each step \cref{mainTheo1} will play a key role. 
\begin{propo} \label{propo6.2} The set $(\nu_{\alpha_0,L}^\Gamma)_{L> 0}$ is precompact, meaning that
	for every $(L_n)_{n\in\BN}\in\BR_+^\BN$ with $L_n \to \infty$, there is a Radon measure $\widehat{\Fm}$ on $\CC(D\Sigma)$ with $$\lim\limits_{i\to\infty}\nu_{\alpha_0,L_{n_i}}^\Gamma=\widehat{\Fm}$$
	for some subsequence $L_{n_i}$,
\end{propo}

\begin{proof}
	The $\nu_{\alpha_0,L_n}^\Gamma$ are measures with support in $\CC(D\Sigma)$ which is locally compact \cite{Bon}. Hence the set of Radon measures on $\CC(D\Sigma)$ has the Heine-Borel property: to show that each sequence has a convergent subsequence it suffices to show that $\{\nu_{\alpha_0,L_n}^\Gamma \}$ is bounded, that is to show that for every continuous and compactly supported function $f$ on $\CC(D\Sigma)$, $\limsup_n\medint\int fd\nu_{\alpha_0,L_n}^\Gamma<\infty$.

	Fix a continuous and compactly supported function $f$ on $\CC(D\Sigma)$. As $f$ has compact support, $|f|$ is  bounded by some $b>0$ and there is some $D>0$ such that $i(\mu,\widehat{\delta_0})\leq 2D$ for every $\mu$ in the support of $f$, where $\delta_0$ is a fixed filling curve of $\Sigma$. Hence 
	\begin{align} \label{equ1}
		\int |f|d\nu_{\alpha_0,L_n}^\Gamma &\leq b\cdot \nu_{\alpha_0,L_n}^\Gamma (\supp(f)) \\
		& \leq b\cdot\nu_{\alpha_0,L_n}^\Gamma (\{\mu \in \CC(D\Sigma)|i(\mu,\widehat{\delta_0})\leq 2D \}).\notag
	\end{align}

	Moreover, we have 
	\begin{align*}
		\nu_{\alpha_0,L_n}^\Gamma (\{\mu \in \CC(D\Sigma)|i(\mu,\widehat{\delta_0})\leq 2D \} )= \frac{\#\{ \alpha\in \Gamma\cdot\alpha_0|i(\widehat{\alpha},\widehat{\delta_0})\leq 2DL_n\}}{L_n^{6g-6+2r}}  = \frac{\#\{ \alpha\in \Gamma\cdot\alpha_0|i(\alpha,\delta_0)\leq DL_n\}}{L_n^{6g-6+2r}}. 
	\end{align*}
	Hence, \cref{mainTheo1} ensures that
	\begin{align*}
\nu_{\alpha_0,L_n}^\Gamma (\{\mu \in \CC(D\Sigma)|i(\mu,\widehat{\delta_0})\leq 2D \}) \xrightarrow[n \to\infty]{ } D^{6g-6+2r} \cdot\Fc^\Gamma_{g,r}(I(\alpha_0))\cdot \Fm_{Thu}^\Sigma(i(\delta_0,\cdot)\leq1),
	\end{align*}
	 which together with \cref{equ1} ensures that $\limsup_n\medint\int fd\nu_{\alpha_0,L_n}^\Gamma<\infty$. That concludes the proof.
\end{proof}

We now want to show that every $\widehat{\Fm}$ as above is supported by $\widehat{\CM\CL}(\Sigma)$. In some sense, this justifies the notation $\widehat{\Fm}$.

\begin{propo} \label{lemma} The measure $\widehat{\Fm}$ in \cref{propo6.2} is supported by $\widehat{\CM\CL}(\Sigma)$.
\end{propo}

\begin{proof}
	In light of \cref{caracterisation}, to show that $\widehat{\Fm}$ has support in $\widehat{\CM\CL}(\Sigma)$ there are three points to prove. We need to prove that $\widehat{\Fm}$ is supported by symmetric measured laminations --- the fact that the elements of the support are symmetric comes from the construction of the $\nu_{\alpha_0,L_n}^\Gamma$ so we just need to show that they are measured laminations. The second point to prove is that they do not cross the image of $\D\Sigma$ in $D\Sigma$. Finally, we need to argue that they do not have connected components of $\D\Sigma$ as leaves.

	Regarding this last point, note that if we assume that the support of $\widehat{\Fm}$ is made of symmetric measured laminations then the elements in the support of $\nu_{\alpha_0,L}^\Gamma$ are all orthogonal to $\D\Sigma$ so the one in the support of $\widehat{\Fm}$ are all transversal to $\D\Sigma$.

Let us now show the first two points. We know that $\widehat{\Fm}$ has support in the symmetric currents so, to show that it is supported by $\CM\CL^\sigma(D\Sigma)$ it suffices to show that for every $R>0$, 
	$$\medint\int_{\{\mu\in\CC(D\Sigma)|i(\mu,\Delta_0)\leq R\}}i(\mu,\mu)d\widehat{\Fm}=0$$
	 where $\Delta_0$ is a filling multicurve of $D\Sigma$. We can assume that $\Delta_0$ decomposes into $\widehat{\delta}_0$ --- where $\delta_0$ is a filling curve of $\Sigma$--- and a multicurve $s$ of $D\Sigma$ that completes $\widehat{\delta}_0$.
	
	For every $L>0$ we have 
	\begin{align}
	\medint\int_{\{\mu\in\CC(D\Sigma)|i(\mu,\Delta_0)\leq R\}}i(\mu,\mu)&d\nu_{\alpha_0,L}^\Gamma= \dfrac{1}{L^{6g-6+2r}}\sum\limits_{\substack{\alpha \in \Gamma\cdot\alpha_0 \\ i(\widehat{\alpha},\Delta_0)\leq LR } } i\left(\frac{\widehat{\alpha}}{L},\frac{\widehat{\alpha}}{L}\right) \label{3.2int} \\
	&\leq \dfrac{1}{L^{6g-6+2r}}\sum\limits_{\substack{\alpha \in \Gamma\cdot\alpha_0 \\ i(\widehat{\alpha},\widehat{\delta_0})\leq LR } } i\left(\frac{\widehat{\alpha}}{L},\frac{\widehat{\alpha}}{L}\right) \notag\\ 
	&=\dfrac{1}{L^{6g-6+2r}}\sum\limits_{\substack{\alpha \in \Gamma\cdot\alpha_0 \\ i(\alpha,\delta_0)\leq LR/2 }  } 2\frac{i(\alpha,\alpha)}{L^2} \notag \\
	&=\left(\frac{R}{2}\right)^{6g-6+2r} \dfrac{2i(\alpha_0,\alpha_0)}{L^2}\dfrac{\#\{\alpha\in\Gamma\cdot\alpha_0|i(\alpha,\delta_0)\leq \frac{LR}{2}\}}{(\frac{LR}{2})^{6g-6+2r}} \notag  \\
	&\xrightarrow[L \to \infty]{ } \left(\frac{R}{2}\right)^{6g-6+2r}\cdot 0\cdot \Fc^\Gamma_{g,r}(\alpha_0)\cdot\Fm_{Thu}^\Sigma(\{i(\cdot,\delta_0)\leq 1\}) \text{ Thm 2.3. }\notag
	\end{align}
	However, there is some $L_n \xrightarrow[n\to\infty]{ } \infty$ such that $\Fm_{\alpha_0,L_n}^\Sigma \xrightarrow[n\to\infty]{} \widehat{\Fm}$. Hence, \cref{3.2int} ensures that  
\begin{equation*}
	\medint\int_{\{\mu\in\CC(D\Sigma)|i(\mu,\Delta_0)\leq R\}}i(\mu,\mu)d\widehat{\Fm} =0,
\end{equation*}
 meaning that the elements of the support of $\widehat{\Fm}$ have no self intersection: they are measured laminations of $D\Sigma$ (see \cref{MLinC}).

The same computations allow us to obtain that for every $R>0$,  $$\medint\int_{\{\mu\in\CC(D\Sigma)|i(\mu,\Delta_0)\leq R\}}i(\mu,\D\Sigma)d\widehat{\Fm}=0.$$ 
Hence the measured laminations in the support of $\widehat{\Fm}$ do not cross the boundary of $\Sigma$ and that concludes the proof.
\end{proof}

\begin{cor} \label{MesuresAssociees}
	Any measure $\widehat{\Fm}$ as in \cref*{propo6.2} arises as the pushforward by the hat operator of a measure $\Fm$ in $\CM\CL(\Sigma)$ defined by
	\begin{equation}\forall U \subset\CM\CL(\Sigma),\quad\Fm(U):=\widehat{\Fm}(\{\widehat{\lambda}|\lambda \in U \}).
	\end{equation}
	\hspace*{16.3cm}\qedsymbol
\end{cor}

In the line of 
\cite{RS}, to show that the limit element is a multiple of the pulled-back Thurston measure we will use the following theorem which comes from \cite{LM}.

\begin{named}{Theorem}[Lindenstrauss-Mirzakhani]
  Let $\mu$ be a locally finite $\Map(\Sigma)$-invariant
	measure on $\CM\CL(\Sigma)$. If for all simple closed curve $\gamma$ of $\Sigma$
\begin{align} \label{mesureNulle}
	 \mu ( \{ \lambda \in \CM\CL(\Sigma) | i(\lambda,\gamma)=0 \} )=0,
\end{align}
	 then $\mu$ is a multiple of the Thurston measure $\Fm_{Thu}^\Sigma$.
\end{named}

It is certainly known to experts that this theorem is also true for a $\Gamma$-invariant measure where $\Gamma$ is a finite index subgroup of $\Map(\Sigma)$. Still, let us give a proof.

\begin{lem} \label{generalisationLM}
	Let $\Gamma$ be a finite index subgroup of $\Map(\Sigma)$.  If $\mu$ is a locally finite $\Gamma$-invariant
	measure on $\CM\CL(\Sigma)$ such that
\begin{align} 
	 \mu ( \{ \lambda \in \CM\CL(\Sigma) | i(\lambda,\gamma)=0 \} )=0
\end{align}
	for every simple closed curve $\gamma$ of $\Sigma$, then $\mu$ is a multiple of the Thurston measure $\Fm_{Thu}^\Sigma$.
\end{lem}

\begin{proof}
	Since every finite index subgroup $\Gamma$ of $\Map(\Sigma)$ admits as a subgroup a finite index normal subgroup of $\Map(\Sigma)$, we can suppose that $\Gamma$ is normal to begin with.

	As the subgroup $\Gamma$ is finite index, we can choose finitely many elements $\phi_1,...,\phi_s$ of $\Map(\Sigma)$ such that every element $\varphi\in\Map(\Sigma)$ can be uniquely written as $\varphi=g\circ\phi_i$ with $g\in\Gamma$. If $\mu$ is a measure as in the statement then we define 
	\begin{equation*}
		\tilde{\mu} := \sum\limits_{i=1}^s \phi_{i*}\mu.
	\end{equation*}
	
	Since $\Gamma$ is normal and $\mu$ is $\Gamma$-invariant, the definition is independant of the choice of the $\phi_i$.

	Now, for $\psi\in\Map(\Sigma)$, $\left[\phi\right]\in\Map(\Sigma)/\Gamma \mapsto \left[\psi\circ\phi\right]\in \Map(\Sigma)/\Gamma$ is well-defined and bijective so
	$\psi_*\tilde{\mu}=\sum\limits_{i=1}^s \psi_*\phi_{i*}\mu
		= \sum\limits_{i=1}^s (\psi\circ\phi_i)_*\mu 
		= \tilde{\mu}$ and $\tilde{\mu}$ is a $\Map(\Sigma)$-invariant locally finite measure on $\CM\CL(\Sigma)$.

	Moreover, if $\gamma$ is a simple curve in $\Sigma$ then \begin{align*}
		\tilde{\mu}(\{\lambda\in\CM\CL(\Sigma)|i(\lambda,\gamma)=0 \})&= \sum\limits_{i=1}^s \phi_{i*}\mu(\{\lambda\in\CM\CL(\Sigma)|i(\lambda,\gamma)= 0\}) \\
		&= \sum\limits_{i=1}^s \mu(\{\lambda\in\CM\CL(\Sigma)|i(\phi_i^{-1}\lambda,\gamma)=0 \}) \\
		&= \sum\limits_{i=1}^s \mu(\{\lambda\in\CM\CL(\Sigma)|i(\lambda,\phi_i\gamma)=0 \}) .
	\end{align*}
	As a consequence, $\tilde{\mu}$ satisfies (\ref{mesureNulle}) as soon as $\mu$ does, and is therefore a multiple of the Thurston measure by Lindenstrauss-Mirzakhani Theorem.  

	Moreover, we can suppose that $\phi_1=Id_\Sigma$ hence $\mu=\phi_{1*}\mu$  is absolutely continuous with respect to $\tilde{\mu}$ and to $\Fm_{Thu}^\Sigma$. However, the Thurston measure is $\Gamma$-ergodic \cite{Masur} so $\mu$ is a positive multiple of $\Fm_{Thu}^\Sigma$.
\end{proof}

\begin{lem} \label{Lemma2}
	If $\widehat{\Fm}$ is as in \cref{propo6.2} then the associated measure $\Fm$ (see \cref{MesuresAssociees}) on $\CM\CL(\Sigma)$ satisfies \cref{mesureNulle}.
\end{lem}

\begin{proof}
	Let $\gamma$ be a simple curve of $\Sigma$, we want to show that $\Fm(\{\lambda\in\CM\CL(\Sigma)|i(\lambda,\gamma)=0 \})=0$. By inner regularity it suffices to show that for every $R>0$, 
	\begin{align} \Fm(\{\lambda\in\CM\CL(\Sigma)| i(\lambda,\gamma)<\epsilon,  i(\lambda,\delta_0)<R \})\xrightarrow[\epsilon\to0]{ } 0 \label{CVG}
	\end{align}
	where $\delta_0$ is a filling curve of $\Sigma$.
	
	Moreover, if $L_n$ is such that $\nu_{\alpha_0,L_n}^\Gamma\xrightarrow[n\to\infty]{ } \widehat{\Fm}$ then the Portmanteau Theorem ensures that
	\begin{align*}
		\Fm(\{\lambda\in\CM\CL(\Sigma)| i(\lambda,\gamma)<\epsilon,  i(\lambda,\delta_0)<R \}) &=\widehat{\Fm}( \{\widehat{\lambda}\in\widehat{\CM\CL}(\Sigma)| i(\lambda,\gamma)<\epsilon,  i(\lambda,\delta_0)<R \}) \\
		&\leq \liminf_n  \nu_{\alpha_0,L_n}^\Gamma ( \{\widehat{\lambda}\in\widehat{\CM\CL}(\Sigma)| i(\lambda,\gamma)<\epsilon,  i(\lambda,\delta_0)<R \}).
	\end{align*}

	Recall that by \cref{kalpha0} $I_{|\Gamma\cdot \alpha_0} : \alpha \in \Gamma\cdot\alpha_0 \mapsto I(\alpha) \in \Gamma\cdot I(\alpha_0)$ is $k_{\alpha_0}$-to-$1$, and that $|i(I(\alpha),\delta)-2i(\alpha,\delta)|\leq 2d_{\alpha_0}i(\delta,\delta) = C_\delta$ for every $\alpha\in\Gamma\cdot\alpha_0$ and $\delta\in\FC_m(\Sigma)$ by \cref{Lemme}. Hence,
	\begin{align*}
		\nu_{\alpha_0,L_n}^\Gamma ( \{\widehat{\lambda}\in\widehat{\CM\CL}(\Sigma)| i(\lambda,\gamma)<\epsilon,&  i(\lambda,\delta_0)<R \}) =\frac{\#\{\alpha\in\Gamma\cdot\alpha_0|i(\alpha,\gamma)<\epsilon L_n, i(\alpha,\delta_0)<R L_n \}}{L_n^{6g-6+2r}} \\
		&\leq \frac{\#\{\alpha\in\Gamma\cdot\alpha_0|i(I(\alpha),\gamma)<2\epsilon L_n, i(I(\alpha),\delta_0)\leq2R L_n + C_{\delta_0} \}}{L_n^{6g-6+2r}} \\
		& \leq k_{\alpha_0} \cdot\frac{\#\{\tau\in\Gamma\cdot I(\alpha_0)|i(\tau,\gamma)<2\epsilon L_n \}}{L_n^{6g-6+2r}} \\
		&\leq k_{\alpha_0}\cdot\nu_{I(\alpha_0),L_n}^\Gamma(\{\lambda|i(\lambda,\gamma)\leq 2\epsilon  \}),
	\end{align*}
where $\nu_{\gamma_0,L_n}^\Gamma=\frac{1}{L^{6g-6+2r}} \sum\limits_{\gamma\in\Gamma\cdot\gamma_0} \delta_{\frac{1}{L}\gamma}$ when $\gamma_0\in\FC_m(\Sigma)$. We get from \cite[Theo. 8.1 or Ex. 8.3]{ES} that $(\nu_{I(\alpha_0),L_n}^\Gamma)_{n \in\BN}$ converges and then

	\begin{align*}
		\liminf_n  \nu_{\alpha_0,L_n}^\Gamma ( \{\widehat{\lambda}\in\widehat{\CM\CL}(\Sigma)| i(\lambda,\gamma)<\epsilon,  i(\lambda,\delta_0)&<R \})  \leq k_{\alpha_0} \cdot\liminf_n \nu_{I(\alpha_0),L_n}^\Gamma(\{\lambda|i(\lambda,\gamma)\leq 2\epsilon  \}) \\
		& \leq k_{\alpha_0}\cdot \limsup_n \nu_{I(\alpha_0),L_n}^\Gamma(\{\lambda|i(\lambda,\gamma)\leq 2\epsilon  \})  \\
		& \leq k_{\alpha_0} \cdot\Fc^\Gamma_{g,r}(I(\alpha_0))\cdot\Fm_{Thu}^\Sigma (\{ \lambda\in\CM\CL(\Sigma)| i(\lambda,\gamma)\leq 2\varepsilon \} ).
	\end{align*}

	All in all, $$\Fm(\{\lambda\in\CM\CL(\Sigma)| i(\lambda,\gamma)<\epsilon,  i(\lambda,\delta_0)<R \}) \leq C\cdot\Fm_{Thu}^\Sigma (\{ \lambda\in\CM\CL(\Sigma)| i(\lambda,\gamma)\leq 2\varepsilon \} )$$ 
	and the Lindenstrauss-Mirzakhani characterisation of the Thurston measure proves \cref{CVG}.  
\end{proof}

We are now able to prove our main theorem:

\begin{theo}\label{Main}
	If $\Sigma$ is a compact connected oriented surface with non-empty boundary and negative Euler characteristic which is not a pair of pants, then for every weighted multiarc $\alpha_0 \in \CA_m(\Sigma)$, and every finite index subgroup $\Gamma$ of $\Map(\Sigma)$, there is $ \Fc^\Gamma_{g,r}(\alpha_0)>0$ such that
	\begin{equation*}
	\lim\limits_{L\to\infty}\nu^\Gamma_{\alpha_0,L} =\Fc^\Gamma_{g,r}(\alpha_0)\cdot\widehat{\Fm}_{Thu}^\Sigma,
	\end{equation*}
	and the convergence occurs with respect to the weak* topology on the set of Radon measures on  $\CC(D\Sigma)$.
\end{theo}

\begin{rmq}
	The constant $\Fc^\Gamma_{g,r}(\alpha_0)$ is the same as in \cref{mainTheo1}.
\end{rmq}

\begin{proof}
	First of all, consider $\widehat{\Fm}$ given by $\widehat{\Fm}=\lim\limits_{n\to\infty}\nu_{\alpha_0,L_n}^\Gamma$. \cref{Lemma2} together with \cref{generalisationLM} ensures that the associated measure $\Fm$ on $\CM\CL(\Sigma)$ is a multiple of the Thurston measure on $\Sigma$ and hence $\widehat{\Fm}=c(L_n)\cdot\widehat{\Fm}_{Thu}^\Sigma$ where $c(L_n)>0$. 

	Let $\delta_0$ be a filling curve of $\Sigma$, the function $i(\delta_0,\cdot)$ is continuous, homogenous and positive on $\CC(\Sigma)$ hence $\Fm_{Thu}^\Sigma(\{i(\delta_0,\cdot)=1 \})=0$. However, $\Fm_{Thu}^\Sigma(\{\lambda\in\CM\CL(\Sigma)|i(\delta_0,\lambda)=1 \})=\widehat{\Fm}_{Thu}^\Sigma(\{\widehat{\lambda}\in\widehat{\CM\CL}(\Sigma)|i(i^+(\delta_0),\widehat{\lambda})=1 \})=\widehat{\Fm}_{Thu}^\Sigma(\D\{\widehat{\lambda}\in\widehat{\CM\CL}(\Sigma)|i(i^+(\delta_0),\widehat{\lambda})\leq1 \})$ and $\CC(D\Sigma)$ is locally compact so by Portmanteau Theorem and \cref{mainTheo1} we obtain the two following results

	\begin{align*}
		\nu_{\alpha_0,L_n}^\Gamma(\{\mu\in\CC(D\Sigma)|i(i^+(\delta_0),\mu)\leq 1 \})&\xrightarrow[n\to\infty]{ } c(L_n)\cdot\widehat{\Fm}_ {Thu}^\Sigma (\{\widehat{\lambda}\in\widehat{\CM\CL}(\Sigma)|i(i^+(\delta_0),\widehat{\lambda})\leq 1 \})\\
		& = c(L_n)\cdot\Fm_ {Thu}^\Sigma (\{\lambda\in\CM\CL(\Sigma)|i(\delta_0,\lambda)\leq 1 \})\\
		& = c(L_n)\cdot\Fm_ {Thu}^\Sigma(\{i(\delta_0,\cdot)\leq 1\}),\\
		\nu_{\alpha_0,L_n}^\Gamma(\{\mu\in\CC(D\Sigma)|i(i^+(\delta_0),\mu)\leq 1 \})&= \frac{1}{L_n^{6g-6+2r}}\#\{\alpha\in\Gamma\cdot\alpha_0 | i(i^+(\delta_0),\widehat{\alpha})\leq L_n \} \\
		& =  \frac{\#\{\alpha\in\Gamma\cdot\alpha_0 | i(\delta_0,\alpha)\leq L_n \} }{L_n^{6g-6+2r}} \\
		& \xrightarrow[n\to\infty]{ } \Fc^\Gamma_{g,r}(\alpha_0)\cdot\Fm_{Thu}^\Sigma(\{i(\delta_0,\cdot)\leq 1\}).
	\end{align*}
	Hence, $C(L_n)=\Fc^\Gamma_{g,r}(\alpha_0)$ does not depend on the sequence $(L_n)_{n\in\BN}$ and whatever the sequence $L_n\xrightarrow[n\to\infty]{ }\infty$, up to passing to a subsequence $n_i$, $$\lim\limits_{i\to\infty} \nu_{\alpha_0,L_{n_i}}^\Gamma = \Fc^\Gamma_{g,r}(\alpha_0)\cdot\widehat{\Fm}_{Thu}^\Sigma.$$
	Since the previous convergence holds for any $L_n\xrightarrow[n \to \infty]{ } \infty$,
	$$	\lim\limits_{L\to\infty} \nu_{\alpha_0,L}^\Gamma =\Fc^\Gamma_{g,r}(\alpha_0)\cdot\widehat{\Fm}_{Thu}^\Sigma.$$
\end{proof}

\section{Application to counting problems}

Armed with \cref{Main} we are now able to focus on counting problems. In this section we are interested in counting the elements in the orbit of a given arc for the action of a finite index subgroup of the mapping class group.

\subsection{Counting bounded arcs}

For $F$ a function on arcs we want to count $\#\{\alpha\in\Gamma\cdot\alpha_0|F(\alpha)\leq L\}$ using  \cref{Main}. To do so, we have to be able to extend $F$ to the currents of $D\Sigma$.

  The more natural examples for $F$ are
	\begin{itemize}
		\item[-] the length function for any Riemannian metric with geodesic boundary on $\Sigma$,
		\item[-] the intersection number with a filling curve $\delta_0$ of $\Sigma$,
		\item[-] the intersection number with a filling current $\mu_0$ of $\Sigma$,
	\end{itemize}
	in that cases, the extension on $\CC(D\Sigma)$ is naturally given by
	\begin{itemize}
		\item[-] the length function associated to the corresponding metric on $D\Sigma$,
		\item[-] the intersection number with $\widehat{\delta_0}$,
		\item[-] the intersection number with $\widehat{\mu_0}$. 
	\end{itemize}
So, we will say that a function $F$ on arcs \textit{extends to currents} if there exists a continuous and homogeneous function on $\CC(D\Sigma)$ whose restriction to the set of arcs is $F$. If this function on currents is also called $F$, it means that $F(\widehat{\alpha})=2F(\alpha)$ for every arc $\alpha$. Since $\CC(\Sigma)\subset \CC(D\Sigma)$ we will say that $F$ is positive if it is a positive function on $\CC(\Sigma)$. Note that for any $\mu\in\CC(\Sigma)$ we also want that $F(\widehat{\mu})=2F(\mu)$.

Now, the same process as in the proof of \cite[Theo 9.1]{ES} allows us to obtain the following Corollary.

\begin{cor} \label{Corollaire}
	Let $\Sigma$ and $\Gamma$ be as in \cref*{Main}. For any weighted multiarc $\alpha_0\in\CA_m(\Sigma)$ and any function $F$ on $\CA_m(\Sigma)$ which extends to a positive function on currents we have 
	$$\lim\limits_{L\to\infty}\frac{\#\{\alpha\in\Gamma\cdot\alpha_0|F(\alpha)\leq L \}}{L^{6g-6+2r}}= \Fc^\Gamma_{g,r}(\alpha_0)\cdot\Fm_{Thu}^\Sigma(\{F(\cdot)\leq 1\}). $$	
\end{cor}

\begin{proof}
	First of all, since $F$ is continuous on $\CC(D\Sigma)$, $$\partial\{ \mu\in\CC(D\Sigma)|F(\mu)\leq 2 \}\subset\{ \mu\in\CC(D\Sigma)|F(\mu)= 2 \}$$ and as $F$ is positive, continuous and homogeneous on $\CC(\Sigma)$
	\begin{align} \label{Bord0}
		\widehat{\Fm}_{Thu}^\Sigma(\{ \mu\in\CC(D\Sigma)|F(\mu)= 2 \})&=\widehat{\Fm}_{Thu}^\Sigma(\{ \widehat{\lambda}\in\widehat{\CM\CL}(\Sigma)|F(\widehat{\lambda})= 2 \}) \\
		& =\Fm_{Thu}^\Sigma(\{ \lambda\in\CM\CL(\Sigma)|F(\widehat{\lambda})= 2 \}) \notag\\
		& =\Fm_{Thu}^\Sigma(\{ \lambda\in\CM\CL(\Sigma)|F(\lambda)= 1 \}) \notag\\
		&=\Fm_{Thu}^\Sigma(\{ \mu\in\CC(\Sigma)|F(\mu)= 1 \}) \notag\\
		&=0 \notag.
	\end{align}
As a consequence, $\CC(D\Sigma)$ being locally compact, the Portmanteau Theorem together with \cref{Bord0} and \cref{Main} ensures that 
\begin{align} \label{CCl}
	\nu_{\alpha_0,L}^\Gamma(\{ \mu\in\CC(D\Sigma)|F(\mu)\leq 2 \}) & \xrightarrow[L \to \infty]{ } \Fc^\Gamma_{g,r}(\alpha_0)\cdot\widehat{\Fm}_{Thu}^\Sigma(\{ \mu\in\CC(D\Sigma)|F(\mu)\leq 2 \}) \\
	& = \Fc^\Gamma_{g,r}(\alpha_0)\cdot\Fm_{Thu}^\Sigma(\{\lambda\in\CM\CL(\Sigma|F(\lambda)\leq 1 \}). \notag
\end{align}
 Moreover, since $F$ is homogeneous we deduce that $$\frac{\#\{\alpha\in\Gamma\cdot\alpha_0|F(\alpha)\leq L \}}{L^{6g-6+2r}}=\frac{\#\{\alpha\in\Gamma\cdot\alpha_0|F(\widehat{\alpha})\leq 2L \}}{L^{6g-6+2r}} = \nu_{\alpha_0,L}^\Gamma(\{ \mu\in\CC(D\Sigma)|F(\mu)\leq 2 \}),$$  and \cref{CCl} concludes the proof.

\end{proof}

\subsection{Counting bi-infinite arcs}

We work now on a non-compact surface $S$ of finite type. More concretely, $S$ has finite genus, finitely many punctures and empty boundary. If $X$ is a fixed finite area hyperbolic structure on $S$ and $\alpha_0$ a bi-infinite arc between two cusps of $S$ we want to determine 
\begin{equation*}
	\#\{\alpha\in\Map(S)\cdot\alpha_0|\ell_X(\alpha)\leq L\}.
\end{equation*}
 To do so, we first have to choose a way to define $\ell_X(\alpha)$. Indeed, with the natural notion of length every bi-infinite arc has infinite length.

 \begin{rmq}
	In the following, by bi-infinite arc we mean bi-infinite geodesic between two cusps of $S$.
 \end{rmq}

 In such a surface one can define the \textit{peripheral self-intersection number} $i_{per}(\gamma,\gamma)$ of a geodesic $\gamma$. This number tells us how much each excursion of the curve into a cusp intersects itself (see \cite[Def. 2.6]{ES} or \cite{MT} for details on the peripheral self-intersection number). The number of self intersections of an excursion being in direct link with the depth reached by this excursion into a cusp \cite{BPT} \cite{MT}, knowing the peripheral self-intersection of a bi-infinite arc we know exactly the maximal depth reached by any finite excursion (it is an excursion that does not leave all compact subsets of the surface) into a cusp's neighborhood. 
 
 \begin{rmq}
	For a weighted multiarc, we define the peripheral self-intersection number as the maximal peripheral self-intersection number of its components.
 \end{rmq}
 In the line of \cite[Prop. 2.2]{MT} we obtain the following lemma.

 \begin{lem}
	Let $S$ be a finite type surface with negative Euler characteristic, no boundary components and finitely many cusps. If $\gamma$ is a bi-infinite arc of $S$ with $i_{per}(\gamma,\gamma)>0$ then the finite excursions of $\gamma$ stay in the compact core of $S$ bounded by the horospheres of length $1/k$ if and only if $i_{per}(\gamma,\gamma)\leq 4k$.
 \end{lem}

 Since the peripheral self-intersection number is stable through the action of $\Map(S)$ we have a natural way to associate a finite length to each infinite arc and that definition will be relevant if we want to count the elements in a given orbit of the mapping class group.

 \begin{rmq}
	We need the notion of length we will define for bi-infinite arcs to be compatible with the length of the measured laminations of the surface. To do so, note that for any hyperbolic metric $X$ on $S$ the support of every $\lambda\in\CM\CL(S)$ is included in $X^1$, the compact core of $X$ bounded by the horospheres of length $1$.
 \end{rmq}

 \begin{definition}
	Let $S$ be a finite type surface with negative Euler characteristic, no boundary and finitely many cusps. For a fixed hyperbolic structure, we define the \textbf{compact length} of a bi-infinite arc $\alpha$ of $S$ by
	\begin{align*}
		\overline{\ell}_X(\alpha) := \left\{
			\begin{array}{ll}
				\ell_X(\gamma\cap{X^{i_{per}(\alpha,\alpha)/4}})\quad &\text{if} \quad {i_{per}(\alpha,\alpha)/4>1} \\
				 \ell_X(\gamma\cap{X^1})\quad &\text{otherwise.}
			\end{array}
		\right.
	\end{align*}

	Where for any $k\geq 1$, $X^k$ is the compact core of $X$ bounded by the embedded horospheres of length $1/k$.
 \end{definition}
	
 \begin{figure}[!ht]
	\centering
	\includegraphics[scale=0.3]{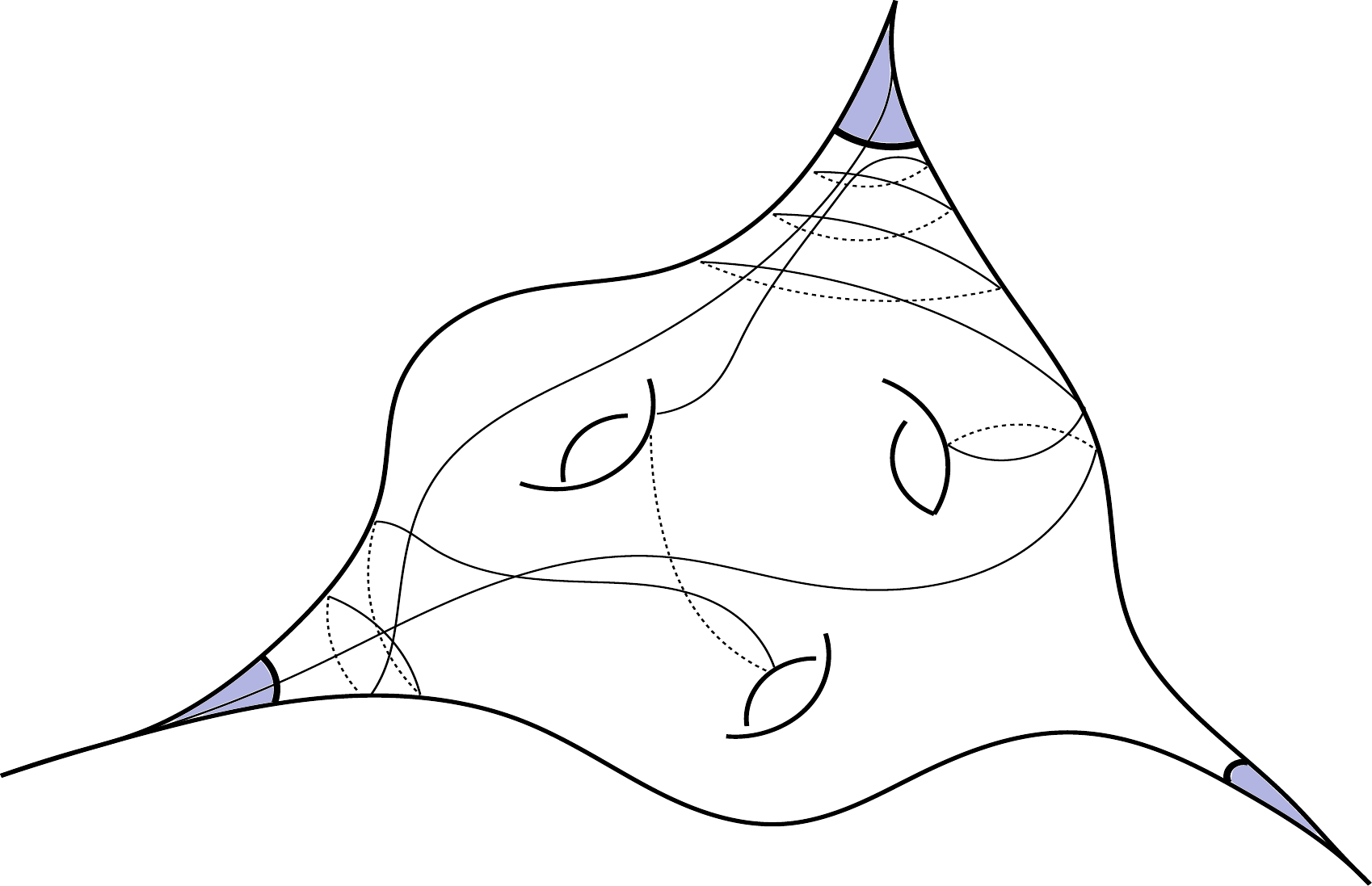}  
	\caption{How to compute $\overline{\ell}$.}
	\label{Troncature}
\end{figure}

 \begin{theo}\label{CuspToCusp}
	Let $S$ be a connected oriented surface with $r>0$ punctures and negative Euler characteristic but not a pair of pants. For any hyperbolic structure $X$ on $S$, if $\alpha_0$ a weighted bi-infinite multiarc and $\Gamma$ is a finite index subgroup of $\Map(S)$ then
	\begin{equation*}
		\lim\limits_{L\to\infty}\frac{\#\{\alpha\in\Gamma\cdot\alpha_0|\overline{\ell}_X(\alpha)\leq L \}}{L^{6g-6+2r}}= \Fc^\Gamma_{g,r}(\gamma_0)\cdot\Fm_{Thu}^S(\{\ell_X(\cdot)\leq 1\}).
	\end{equation*}
 \end{theo}

 \begin{proof}
	If $S$ has genus $g$ and $r$ cusps then we call $\Sigma$ the compact surface of genus $g$ with $r$ boundary components. From nowon $X$ is a fixed hyperbolic structure on $S$ and we want to construct a metric on $\Sigma$ from $X$.

	Fix the bi-infinite multiarc $\alpha_0$, there is $k>0$ such that $\overline{\ell}_X(\alpha) := \ell_X(\gamma\cap{X^k})$. If we cut $S$ along the embedded horospheres of length $1/k$ then we obtain a $CAT(-1)$ metric structure on $\widetilde{\Sigma}$ (see \cite[Ex. 1.16 p168]{BH}) for which the horosphere boundaries are geodesic, hence the associated gluing metric on $D\Sigma$ given by the corresponding length function $\ell_{D\Sigma}$ is also $CAT(-1)$ (see \cite[Theo. 11.1 p347]{BH}) on $\widetilde{D\Sigma}$.

	In a $CAT(-1)$ space the length and the stable length coincide hence the length of curves $\ell_{D\Sigma}$ coming from $X$ is equal to the stable length for the action $\pi_1(D\Sigma)\curvearrowright \widetilde{D\Sigma}$. However, the stable length for any discrete and cocompact isometric action of a torsion-free hyperbolic group on a geodesic metric space extends to a continuous, positive and homogeneous function on currents (see \cite[Theo. 1.5]{EPS}). Hence, \cref{Corollaire} applied with $F=\overline{\ell}_X$ which extends to $\ell_{D\Sigma}$ and for the measured laminations the different notions of length coincide and that concludes the proof.
 \end{proof}

 \begin{rmq}
	There are many ways to decide how to truncate an infinite arc in order to take an interest in its length. See for example \cite{Bell22} or \cite{Parlier} for other ways to do so.
	For example, $\ell^t_X$ is the length function on infinite arcs such that the length of a cusps-to-cusps arc is the length of this arc beteween the first time it enters $X^t$ and the last time it leaves it. The advantage of this definition is that it does not depend on the chosen arc and is more visual in the universal cover (\cref*{TroncatureT} shows how to see $\ell^t_X$ in the universal cover).

	\begin{figure}[!ht]
		\centering
		\includegraphics[scale=0.4]{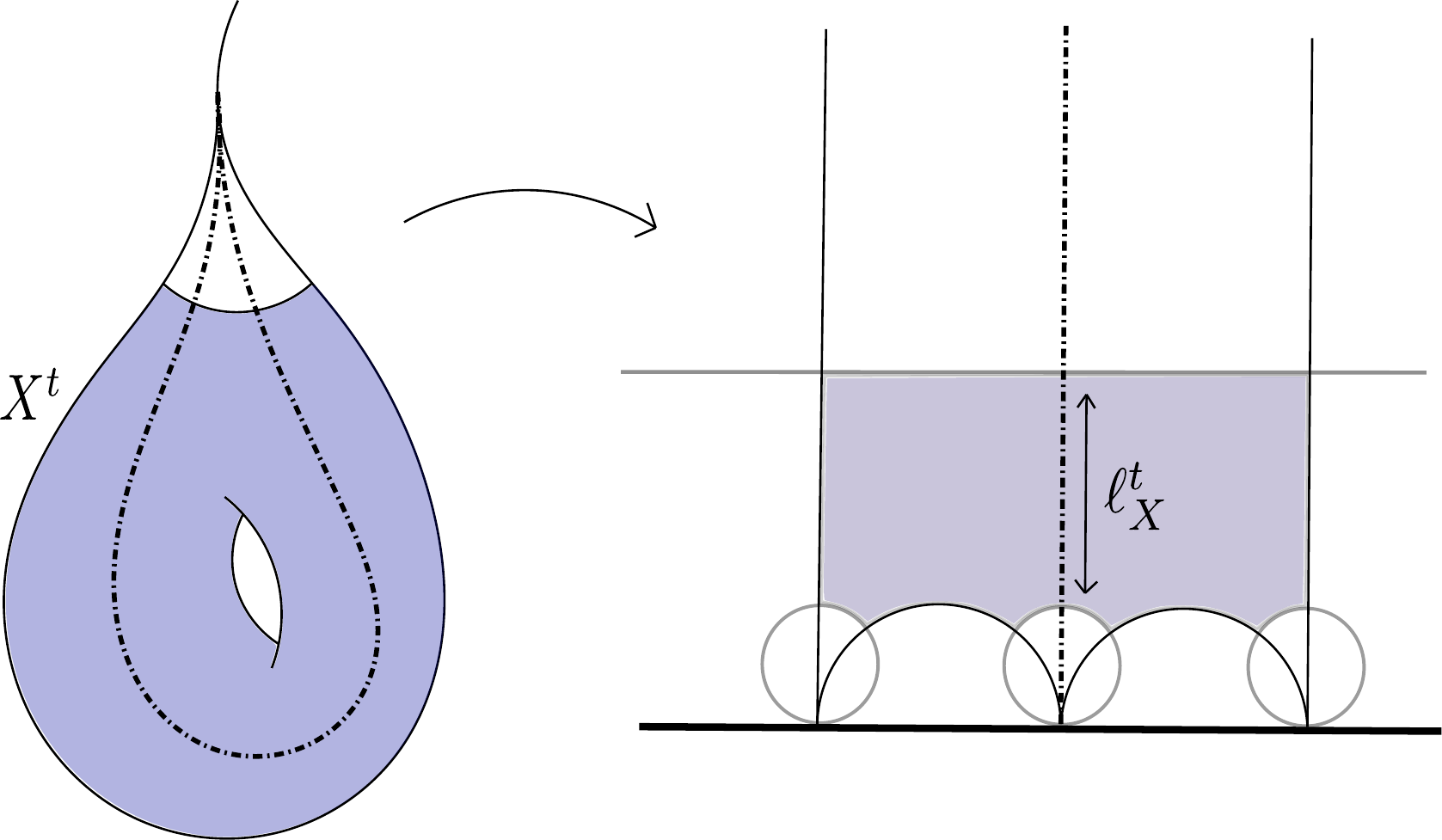}  
		\caption{How to compute $\ell_X^t$}
		\label{TroncatureT}
	\end{figure}
	
	 This notion of length differs from $\overline{\ell}_X$ by a constant hence as an immediate corollary of \cref{CuspToCusp} we have that for all $t\geq 1$ and any infinite weighted multiarc $\alpha_0$
	\begin{equation*}
		\lim\limits_{L\to\infty}\frac{\#\{\alpha\in\Map(\Sigma)\cdot\alpha_0|\ell^t_X(\alpha)\leq L \}}{L^{6g-6+2r}}= \Fc^\Gamma_{g,r}(\alpha_0)\cdot\Fm_{Thu}^S(\{\ell_X(\cdot)\leq 1\}).
	\end{equation*}
 \end{rmq}

 \subsection{Counting arcs on orbifolds}
We now work on a compact orientable orbifold $O$ rather than on $\Sigma$ or $S$. We denote by $g$ its genus and $r$ the number of boundary components and singularities, assuming that it has non-empty boundary. As for surfaces, we will assume that $(g,r)\neq (0,3)$.

One can define $\CC^{or}(O)$ the set of geodesic currents for $O$, and a notion of Thurston measure in $\CC^{or}(O)$ (see \cite{ES3}). In the line of the known results for curves it appears \cite{ES3} that for every $\gamma_0\in\FC_m(O)$ and $\Gamma$ finite index subgroup of $\Map^{or}(O)$ there is a positive constant $\Fc^\Gamma_{g,r}(\gamma_0)$ such that 
\begin{align} \label{CvgMesureOrbifold}
	\lim\limits_{L\to\infty}\frac{1}{L^{6g-6+2r}}\sum\limits_{\gamma\in\Gamma\cdot\gamma_0}\delta_{\frac{1}{L}\gamma} = \Fc^\Gamma_{g,r}(\gamma_0)\cdot \Fm_{Thu}^O,
\end{align}
 where the convergence occurs with respect to the weak* topology on the set of Radon measures on the set of geodesic currents of $O$.
As a consequence, for every continuous, homogeneous and positive function $F$ on $\CC^{or}(O)$,
\begin{align} \label{CountingCurvesOrbifolds}
	\lim\limits_{L\to\infty}\frac{\#\{\gamma\in\Gamma\cdot\gamma_0|F(\gamma)\leq L \}}{L^{6g-6+2r}}= \Fc^\Gamma_{g,r}(\gamma_0)\cdot\Fm_{Thu}^O(\{F(\cdot)\leq 1\}). 
\end{align}

This naturally raises the question of applying the results of this paper to the case of orbifolds: 
\begin{enumerate}
	\item Fuschian groups are LERF \cite{scott} so \cref{Lemme} is still true, 
	\item \cref{CountingCurvesOrbifolds} ensures that we still have \cref{mainTheo1} for orbifolds,
	\item for a compact orbifold the set of geodesisc currents is still locally compact \cite[Section 4.1]{ES3} so with the same proof as in the case of surfaces, \cref{propo6.2} happens in the orbifold case,
	\item the same caracterisations of measured laminations as in the surface case are true for orbifolds what ensures that \cref{lemma} is still true,
	\item the Thurston measure on $O$ can be seen as the pushforward for some application of the Thurston measure on the surface associated to $O$ \cite[Lem. 4.1]{ES3}, which ensures that Lindenstrauss-Mirzakhani characterisation of the Thurston measure and \cref{generalisationLM} are true for orbifolds,
	\item finally, \cref{CvgMesureOrbifold} implies that we are able to prove \cref{Lemma2} for $O$.
\end{enumerate}
All the constructions of this paper apply in the orbifold case which gives us a version of \cref{Main} and \cref{Corollaire} for orbifolds.

\begin{theo}\label{MainOrbifold}
	
	If $O$ is a compact, connected, oriented orbifold with non-empty boundary such that $(g,r)\neq(0,3)$, and $\Gamma$ is a finite index subgroup of $\Map^{or}(O)$ then for every $\alpha_0 \in \CA_m(O)$ weighted multiarc 
\begin{equation*}
	\lim\limits_{L\to\infty}\nu^\Gamma_{\alpha_0,L} = \Fc^\Gamma_{g,r}(\alpha_0)\cdot\widehat{\Fm}_{Thu}^O.
\end{equation*}
The convergence occurs with respect to the weak* topology on the set of Radon measures on $\CC^{or}(DO)$ and $\Fc^\Gamma_{g,r}(\alpha_0)$ is a constant comming from \cref{mainTheo1} and \cite{ES3}.
\end{theo}

\begin{cor} \label{CorollaireOrbifold}
	With the same conditions as above, for any function $F$ on $\CA_m(O)$ which extends to a positive function on currents we have 
	$$\lim\limits_{L\to\infty}\frac{\#\{\alpha\in\Gamma\cdot\alpha_0|F(\alpha)\leq L \}}{L^{6g-6+2r}}= \Fc^\Gamma_{g,r}(\alpha_0)\cdot\Fm_{Thu}^O(\{F(\cdot)\leq 1\}). $$
\end{cor}

Here, the notion of extension of a function is the same as in \cref*{Corollaire}. In the line of what we have done for surfaces, it is also possible to count bi-infinite arcs in non-compact orbifolds.

\bibliographystyle{plain}
\bibliography{bibliography.bib}

\end{document}